\documentclass[12pt]{amsart}
\usepackage{amssymb,latexsym,amsfonts}
\usepackage{amsmath}
\usepackage{amscd}
\usepackage{pb-diagram,pb-xy}

\usepackage{graphicx}
\usepackage{hyperref}
\usepackage{color}
\usepackage[matrix,arrow]{xy}
\def\ov{\stackrel}
\newcommand{\lra}{\longrightarrow}

\DeclareMathOperator*{\so}{so}
\DeclareMathOperator*{\spin}{spin}
\newcommand{\Ker}{{\rm Ker}}
\newcommand{\Hom}{{\rm Hom}}
\newcommand{\Scal}{{\rm Scal}}
\newcommand{\tr}{\operatorname{tr}}
\newcommand{\tmf}{{\rm tmf}}
\newcommand{\Imm}{{\rm Im}}
\newcommand{\Ric}{{\rm Ric}}
\newcommand{\Fbr}{{\rm Fivebrane}}
\newcommand{\str}{{\rm string}}
\DeclareMathAlphabet{\mathpzc}{OT1}{pzc}{m}{it}

\def\p{\partial}
\def\st{\stackrel}
\def\R{\mathbf R}
\def\P{\mathbf P}
\def\H{\mathbf H}
\def\Ca{{\mathbb C}{\mathbf a}}
\def\Z{{\mathbf Z}}
\def\C{{\mathbf C}}
\def\Q{{\mathbf Q}}
\def\rh{\hookrightarrow}
\newtheorem{theorem}{Theorem}[section]
\newtheorem{corollary}[theorem]{Corollary}
\newtheorem{lemma}[theorem]{Lemma}
\newtheorem{proposition}[theorem]{Proposition}
\newtheorem*{thm-a}{Theorem A}
\newtheorem*{thm-b}{Theorem B}
\newtheorem*{thm-b'}{Theorem B$'$}
\newtheorem*{thm-c}{Corollary B}
\newtheorem*{thm-c'}{Theorem A$'$}
\newtheorem*{conj-d}{Conjecture D}
\newtheorem*{conj-c}{Conjecture C}
\newtheorem*{conj-s}{Stolz Conjecture (1996)}
\newtheorem{thm}{Theorem}[section]
\theoremstyle{definition}

\newtheorem*{remark}{Remark}

\newtheorem{remarks}[thm]{Remarks}

\begin{document}

\author{Boris Botvinnik} 
\author{Mohammed Labbi} 
\title{Highly connected manifolds of positive 
$p$-curvature}
\date{}

\maketitle
\begin{abstract} 
  We study and in some cases classify highly connected manifolds which
  admit a Riemannian metric with positive $p$-curvature. The
  $p$-curvature was defined and studied by the second author in
  \cite{Labbi1,Labbi1a, Labbigroups}. It turns out that the  positivity of
  $p$-curvature could be preserved under surgeries of codimension
  at least $p+3$. This gives a key to reduce a geometrical classification
  problem to a topological one, in terms of relevant bordism groups
  and index theory. In particular, we classify $3$-connected manifolds
  with positive $2$-curvature in terms of the bordism groups
  $\Omega^{\spin}_*$, $\Omega^{\str}_*$, and by means of
  $\alpha$-invariant and Witten genus $\phi_W$. Here we use results
  from \cite{Dessai}, which provide appropriate generators of the ring
  $\Omega^{\str}_*\otimes\Q$ in terms of ``geometric
  $\Ca\P^2$-bundles'', where the Cayley projective
    plane $\Ca\P^2$ is a fiber and the structure group is $F_4$ which is
    the isometry group of the standard metric on $\Ca\P^2$.
\end{abstract}
 
\tableofcontents

\section{Introduction and statement of the main results}\label{intr} 
\subsection{Positive scalar curvature} There is a fundamental result
due to Gromov and Lawson \cite{GL1}, Schoen and Yau \cite{Schoen-Yau}
known as the ``Surgery Theorem''.  It shows that positivity of the scalar
curvature can be preserved under surgery of codimension at least
three.  In particular, the surgery technique provides a key to
classifying  simply-connected manifolds admitting a metric with positive
scalar curvature, \cite{GL1}, \cite{Stolz1}. These are the results:
 \begin{theorem}\label{GL-thm1}
{\rm \cite[Theorem A]{GL1}} Let $M$ be a compact  non-spin
simply-connected manifold with $\dim M=n\geq 5$. Then
$M$ always admits a metric $g$ with positive scalar curvature.
  \end{theorem}
Let $ \alpha :
  \Omega^{\spin}_n\to KO_n $ be the Atiyah-Bott-Shapiro homomorphism
  which evaluates the index of the Dirac
  operator on a spin manifold $M$ representing a cobordism class
  $[M]\in \Omega^{\spin}_n$.
  \begin{theorem}\label{GL-thm2}
{\rm \cite[Theorem B]{GL1}, \cite{Stolz1}} Let $M$ be
a compact spin simply-connected manifold with $\dim
M=n\geq 5$. Then $M$ admits a metric $g$ with positive scalar
curvature if and only if $\alpha([M])=0$ in the group $KO_n$.
  \end{theorem}
  It turns out that there are many other Riemannian invariants that
  are also stable under some type of surgeries, see, for example,
  \cite{ADH,BD,MPU,Petean,Wraith}.  Among such invariants are
  $p$-curvature $s_p$ and the second Gauss-Bonnet curvature which were
  studied by the second author, see \cite{Labbi1,Labbi2}.
\subsection{Positive $p$-curvature}
Let $(M,g)$ be a Riemannian manifold, and $TM$ be the tangent bundle.
We denote by $G_p(TM)$ the bundle of Grassmanians of $p$-dimensional
subspaces of the tangent bundle $TM$. Then the $p$-curvature $s_p$ is
defined as a function $s_p: G_p(TM)\to \R$ as follows. For a
$p$-dimensional space $V\in G_p(TM_x)$, the value of $s_p(V)$ is a
``partial trace'' of the curvature tensor, along all directions,
perpendicular to the subspace $V\subset TM_x$, see \cite{Labbi1} and
section \ref{section3} for details.  The curvature
$s_0$ is nothing but the scalar curvature $\Scal$, furthermore, the
function $s_{p-1}$ can be thought as an appropriate trace of $s_p$.
In particular, positivity of $s_p$ implies positivity of the
curvatures $s_j$ for all $j<p$, including the scalar curvature. It
turns out that positivity of the $p$-curvature $s_p$ is also stable
under surgeries of codimension at least $3+p$, see \cite[Main
Theorem]{Labbi1}.

The surgery result \cite[Main Theorem]{Labbi1} gives an appropriate
setup to classifying manifolds admitting a metric with positive
$p$-curvature for $p\geq 1$ similar to the case of positive scalar
curvature. The first interesting case is when $p=1$. Then the
curvature $s_1$ coincides (up to a factor 2) with the quadratic form
associated to the $2$-tensor $S$ defined by
$$
\begin{array}{c}
S_{ij}=\frac{1}{2}\Scal \cdot g_{ij} - \Ric_{ij}.
\end{array}
$$
The tensor $-S$ is also known as the {\sl Einstein tensor}, and the
1-curvature $s_1$ is called the {\sl Einstein curvature}, see
\cite{Labbi1} and \cite{Labbi1a}. We notice that
$$
\begin{array}{c}
  \tr S = \frac{(n-2)}{2} \Scal \ .
\end{array}
$$
Thus positivity of the tensor $S$ is the same as positivity of the
curvature $s_1$, and positivity of $s_1$ implies positivity of the
scalar curvature for $n\geq 3$. 

An interesting case here is when the manifolds in question are
$2$-connected. Then such manifolds are necessarily spin-manifolds, and
the relevant cobordism group is $\Omega^{\spin}_n$. Here is the
classification result analogous to Theorem \ref{GL-thm2}:
\begin{theorem}\label{Labbi-thm2} {\rm
      \cite[Theorem I]{Labbi1}} Let $M$ be a compact $2$-connected manifold
    with  $\dim M=n\geq 7$. Then $M$ admits a metric
    $g$ with positive $1$-curvature if and only if $\alpha([M])=0$ in
    the group $KO_n$.
  \end{theorem}

The main technique in proving Theorem
  \ref{Labbi-thm2}, is a Surgery Theorem \cite[Main Theorem]{Labbi1}
  and the results by S. Stolz on {\sl geometric $\H\P^2$-bundles}. 
  \subsection{Geometric $\H\P^2$-bundles} We recall that in order to
  prove that vanishing of the index $\alpha([M])\in KO_n$ is
  sufficient for existence of a metric with positive scalar curvature
  on $M$, S. Stolz proves that all cobordism classes in $\ker
  \alpha\subset \Omega^{\spin}_n$ can be realized as total spaces of
  geometric $\H\P^2$-bundles.

  In more detail, let $PSp(3)$ be the projectivization of the
  symplectic orthogonal group $Sp(3)$. It is well-known that the group
  $PSp(3)$ is the isometry group of the standard metric on the
  projective plane $\H\P^2$. Let $BPSp(3)$ be the classifying space of
  the group $PSp(3)$, and $EPSp(3)\to BPSp(3)$ be the universal
  principal bundle. This gives a {\sl universal geometric
    $\H\P^2$-bundle} $E(\H\P^2)\to BPSp(3)$ with a fiber $\H\P^2$ and
  a structure group $PSp(3)$, where the total space $E(\H\P^2)$ is
  defined in a usual way:
$$
E(\H\P^2)= EPSp(3)\times_{PSp(3)}\H\P^2.
$$
Then for any map $f: B\to BPSp(3)$, there is a natural pull-back
$\H\P^2$-bundle $E\to B$ given by the diagram:
$$
\begin{diagram}
\setlength{\dgARROWLENGTH}{1.2em}
\node{E}
      \arrow{s}       
      \arrow[2]{e,t}{\hat f}
\node[2]{E(\H\P^2)}
      \arrow{s}
\\
\node{B}
      \arrow[2]{e,t}{f}
\node[2]{BPSp(3)}
\end{diagram}
$$
This construction defines a transfer map
$$
T : \Omega_{n-8}^{\spin}(BPSp(3)) \lra \Omega^{\spin}_n
$$
which takes a cobordism class of a map $f: B\to BPSp(3)$ to the
cobordism class of the manifold $E$ as above. The following result
provides a key in  proving the necessity in Theorem \ref{GL-thm2}.
\begin{theorem}\label{stolz}{\rm (Stolz, \cite{Stolz1})}
There is an isomorphism $\Imm \ T\cong \ker \alpha$.
\end{theorem}
By construction, the total space $E$ of a geometric $\H\P^2$-bundle
carries a metric with positive scalar curvature, which is given by a
choice of any metric on the base and giving a standard homogeneous
metric to each fiber $\H\P^2$ scaled appropriately to get a positive
scalar curvature on $E$. 

One can observe that if $E$ is the total space of a geometric
$\H\P^2$-bundle, then it carries a metric with positive curvature
$s_p$ for $p\leq 6$.

\subsection{Main results}
Assume that $M$ is a $3$-connected manifold. Then $M$ has a canonical
spin-structure. There are two possibilities: either $M$ is
string-manifold or not. It is well-known that the obstruction to
existence of a string-structure
is given by $\frac{1}{2}p_1(M)$, where $p_1$ is the first Pontryagin
class. The following result is somewhat analogous to Theorem
\ref{GL-thm1}:
\begin{thm-a}\label{main-1} 
  Let $M$ be a compact $3$-connected
  non-string manifold with $\dim M=n\geq 9$. Then $M$ admits a
  Riemannian metric $g$ with positive $2$-curvature if and only if
  $\alpha([M])=0$ in the group $KO_n$, where $\alpha:
  \Omega^{\spin}_n\to KO_n$ is as above.
\end{thm-a}
Any $3$-connected manifold is $\spin$, and $\Omega^{\spin}_*$ is a
relevant bordism group here. We use the surgery technique and Theorem
\ref{stolz} to show that if $\alpha([M])=0$ and $M$ is not string,
then it has a metric with positive $2$-curvature which is ``pulled
back'' from a nice metric on the total space of a geometric
$\H\P^2$-bundle as above.


Let now $M$ be $3$-connected and string. A relevant bordism group
here is $\Omega^{\str}_*$.  Precisely, we prove the following theorem
which is analogous to Theorem B of \cite{GL1}.
\begin{thm-b}\label{main-b} 
 Let $M_1$ be a compact $(3+r)$-connected, $0\leq r\leq 3$, string
  manifold of dimension $n\geq 9+2r$. Assume that $[M_1]=[M_0]$ in the
  cobordism group $\Omega_n^{\str}$, where $M_0$ admits a metric $g_0$
  with $s_{r+2}(g_0)>0$. Then $M_1$ also admits a metric $g_1$
  with $s_{r+2}(g_1)>0$.

  In particular, a compact $3$-connected string manifold $M$ of
  dimension $n\geq 9$ that is string cobordant to a manifold of
  positive $2$-curvature admits a metric with positive $2$-curvature.
\end{thm-b}
For instance, if $M$ is string cobordant to zero,
then the conclusion of the theorem holds for $M$. It
  is known that $\Omega_n^{\str}=0$ for $n=11$ or $n=13$; therefore
any compact $3$-connected string manifold of dimension $11$ or $13$
always has a metric with positive $2$-curvature.

Let $I$  denote the subset of
$\Omega^{\str}_*$ which consists of bordism classes containing
representatives with positive $2$-curvature. Clearly
  $I$ is an ideal of $\Omega^{\str}_*$  since the cartesian product of
a manifold of positive $2$-curvature with an arbitrary manifold has
positive $2$-curvature.
We therefore define the following {\sl geometrical genus}:
\begin{equation*}
\Pi: \Omega^{\str}_*\rightarrow \Omega^{\str}_*/I,
\end{equation*}
which is a ring homomorphism. 

 Let $\phi_W : \Omega^{\str}_*\to \Z[[q]]$ be the Witten genus, see
 \cite{Dessai,Stolz}, and section \ref{witten-genus} below.  By
 definition, if $\phi_Wx\neq 0$, then $x\in \Omega^{\str}_*$ has
 infinite order.  
We prove the following result which is analogous to
 Corollary B of \cite{GL1}. 
\begin{thm-c}\label{main-2}
  Let $N$ be a $(3+r)$-connected, for $0\leq r\leq 3$, string manifold
  of dimension at least $9+2r$ with vanishing Witten genus then some
  multiple $N\sharp \cdots \sharp N$ carries a metric
  of positive $(r+2)$-curvature.

  In particular, if $N$ is a $3$-connected string manifold of
  dimension at least $9$ with vanishing Witten genus then some
  multiple $N\sharp \cdots \sharp N$ carries a metric
  of positive $2$-curvature.
\end{thm-c}
This result suggests that the geometric genus $\Pi$ is
  related to the Witten genus.  It is an open question whether $N$
itself carries a metric of positive $2$-curvature.

Clearly, Theorems A, B and Corollary B give only partial
classification of manifolds with metrics of positive
$2$-curvature. However, we use a construction which
eventually may be useful in obtaining  an affirmative classification.
Before stating our conjecture, we briefly describe
the construction.

Let $F_4$ be the 52-dimensional compact simple sporadic Lie group. It
is well-known that it contains a closed subgroup isomorphic to
$Spin(9)$ which is unique up to inner automorphism.  We denote by
$\Ca\P^2$ the Cayley projective plane which coincides with the
homogeneous space $F_4/Spin(9)$. Then the canonical homogeneous metric
on $\Ca\P^2$ has $F_4$ as a full isometry group, see
\cite[p. 264]{Wolf}. Let $BF_4$ be a classifying space, and $EF_4\to
BF_4$ be a universal principle $F_4$-bundle. A universal
{\sl geometrical $\Ca\P^2$-bundle} can be identified with the fiber
bundle $BSpin(9)\to BF_4$ which has a fiber $\Ca\P^2$ and a structure
group $F_4$. Then for a manifold $L$ and a map $f : L\to BF_4$, we
obtain the following map of fiber bundles
\begin{equation}\label{pull-back}
\begin{diagram}
\setlength{\dgARROWLENGTH}{1.5em}
\node{W}
     \arrow[2]{e,t}{f^*}
     \arrow{s,r}{\pi}
\node[2]{BSpin(9)}
     \arrow{s}
\\
\node{L}
     \arrow[2]{e,t}{f}
\node[2]{BF_4}
\end{diagram}
\end{equation}
The fiber bundle $\pi : W\to L$ as above is called a {\sl geometrical
  $\Ca\P^2$-bundle}. 

It is well-known that $\Omega^{\str}_*\otimes \Q$ is a polynomial
ring.  In more detail, A. Dessai shows that there exist generators
$x_{4k}$, such that
$$
\Omega^{\str}_*\otimes \Q\cong \Q[x_8,x_{12},x_{16},\ldots] \ ,
$$
and each element $x_{4k}$ with $k\geq 4$ is represented by a manifold
$W^{4k}$ which is a total space of a geometrical
$\Ca\P^2$-bundle $\pi_k : W^{4k}\to L^{4k-16}$, see \cite{Dessai} and
section \ref{witten-genus} below.  We consider a transfer map
$$
T^{\str} : \Omega^{\str}_{\ell}(BF_4)\lra \Omega_{\ell+16}^{\str}
$$ 
given as follows. Let $f: L\to BF_4$ be a map representing an element
$x\in \Omega^{\str}_{\ell}(BF_4)$. Then the manifold $W$ from
(\ref{pull-back}) represents the element $T^{\str}(x)\in
\Omega_{\ell+16}^{\str}$. Also, we recall that there is an integral
version of the Witten genus
$$
\phi_W^{\Z}:
\Omega_*^{\str}\to KO_*[[q]]
$$ 
which factors through the coefficients $\tmf_*$ of the
{\sl topological modular forms theory} $\tmf$ (formally known as
$eo_2$):
\begin{equation}\label{pull-back2}
\begin{diagram}
\setlength{\dgARROWLENGTH}{1.5em}
\node{\Omega_*^{\str}}
     \arrow[2]{e,t}{\phi_W^{\Z}}
     \arrow{se,r}{\phi_{AHR}}
\node[2]{KO_*[[q]]}
\\
\node[2]{\tmf_*}
     \arrow{ne,b}{\omega}
\end{diagram}
\end{equation}
Here $\phi_{AHR}: \Omega_*^{\str}\to \tmf_*$ is the string-orientation
constructed by Ando, Hopkins and Strickland, see \cite{AHS1,AHS2}.
\begin{remark} It is known that the groups
    $\Omega_*^{\str}$ have no $p$-torsion away from $p=2,3$. It is
    tempting to conjecture that $\Imm \ T^{\str}$ and $\Ker \
    \phi_{AHR}$ coincide in $\Omega_*^{\str}$ localized at primes 2
    and 3. It turns out, this is too optimistic: the authors were
    informed by M. Joachim that the image $\Imm \ T^{\str}$ is
    strictly less than $\Ker \ \phi_{AHR}$ in dimension
    32. Nevertheless, we think that one may use other homogeneous
    spaces, besides $\Ca\P^2$, to represent elements of the kernel
    $\Ker \ \phi_{AHR}$ by manifolds with positive 2-curvature.
\end{remark}

\begin{conj-c}
  Let  $M$ be a $3$-connected string manifold with
  $\dim M =n \geq 9$. Then $M$ admits a Riemannian metric of
  positive $2$-curvature if and only if $\phi_{AHR}([M])=0$ in $\tmf_n$.
\end{conj-c}
We note that Conjecture C  is weaker than Stolz'
conjecture \cite[Conjecture 1.1]{Stolz} on the existence of a metric
with positive Ricci curvature. However it seems that it is
still very difficult to verify.
\subsection{Generalizations}
The previous results are generalized in this paper in 
different directions.

On one hand, we show that all the previous theorems and conjectures
are still valid if one replaces everywhere positive $2$-curvature
$s_2$ by positive second Gauss-Bonnet curvature $h_4$ or by both
$s_2>0$ and $h_4>0$. Recall that the $h_4$ curvature is a scalar
function defined on the manifold that generalizes the usual scalar
curvature. It is shown in \cite{Labbi2} that it is preserved under
surgeries of codimension at least $5$.

On the other hand, we prove that similar results hold for
$3$-curvature $s_3$ in the frame of $4$-connected Fivebrane and non
Fivebrane manifolds. Recall that a Fivebrane manifold is a string
manifold for which the fractional pontryagin class $\frac{1}{6}p_2$
vanishes.
 
The above Corollary B asserts in particular that if a compact
$6$-connected manifold $N$ is with dimension $\geq 15$ and with
vanishing Witten genus then some multiple of it $N\sharp \cdots\sharp
N$ carries a metric of positive $5$-curvature. We prove the following
analogous of Theorem A in this context:
%
\begin{thm-c'}
 Let $N$ be a $7$-connected and non-Fivebrane compact manifold of
  dimension $\geq 15$. If $N$ is string-cobordant to a manifold $M$
  which carries a metric with positive $6$-curvature, then
  $N$ also carries a metric with positive $6$-curvature.

  In particular, if a compact non-Fivebrane $7$-connected manifold $N$
  of dimension $\geq 15$  has a vanishing Witten genus then some
  multiple of it $N\sharp \cdots\sharp N$ carries a metric of positive
  $6$-curvature.
\end{thm-c'}
It remains an open question to prove that $N$ itself carries a metric
of positive $6$-curvature.

From another prospective, we prove the following
generalization of Theorem B:
\begin{thm-b'}
  Let $M_1$ be a compact $(4+r)$-connected, $0\leq r\leq 3$, Fivebrane
  manifold of dimension $n\geq 11+2r$. Assume that $[M_1]=[M_0]$ in
  the cobordism group $\Omega_n^{\Fbr}$, where $M_0$ admits a metric
  $g_0$ with $s_{3+r}(g_0)>0$. Then $M_1$ also admits a metric $g_1$
  with $s_{3+r}(g_1)>0$.

  In particular, a compact $4$-connected Fivebrane manifold $M$ with
  $\dim M=n\geq 11$ that is Fivebrane cobordant to a manifold of
  positive $3$-curvature also carries a metric with positive
  $3$-curvature.
\end{thm-b'}
The paper also  contains further generalizations of
the previous results, whenever it is appropriate, to all higher
$p$-curvatures in the case  of highly connected
$BO\!\left<\ell\right>$-manifolds.

\subsection{Plan of the paper} 
Section 2 contains basic definitions of string and Fivebrane
manifolds, string and Fivebrane cobordism rings and more general
$BO\!\left<\ell\right>$-manifolds and the corresponding cobordism
rings.

In sections 3 and 4, we prepare for the proof of the main results. In
section 3, we study some interactions between the codimension size of
a surgery made within a given $BO\!\left<\ell\right>$-cobordism class
and the order of connectivity of representatives of that class.  In
section 4, we recall the definitions of $p$-curvatures $s_p$ and the
second Gauss-Bonnet curvature $h_4$. We emphasize that the most
important property of these curvatures is   the stability of
their positivity under surgeries of sufficiently high codimensions.

In section 5, we prove Theorems A, B and B$'$. In section 6 we recall
useful material about the Witten genus and the recent results of
A.  Dessai about the rational cobordism groups and the
kernel of the Witten genus. The results of section 6
are used in section 7 in proving Theorem A$'$ and Corollary B.
\subsection{Acknowledgements} The first author
  thanks Michael Joachim and Anand Dessai and IH\`ES for
    a partial financial support. The second author acknowledges a
    financial support from the Deanship of Scientific Research at the
    University of Bahrain (ref. 2011/13).

\section{String and $BO\!\left<\ell\right>$-cobordism: Basic
  definitions}\label{section1}
Let $\R^{n+k}$ be the Euclidian space. 
We denote by 
$G_n(\R^{n+k})$ the Grassmanian manifold of
$n$-dimensional subspaces of $\R^{n+k}$, and by 
$$
U_{k,n}\lra
G_n(\R^{n+k}) \ \ \ \mbox{and} \ \ \ U_{k,n}^{\perp}\lra G_n(\R^{n+k})
$$ 
the tautological  bundle and  its orthogonal complement respectively. 
Then one obtains the spaces
$$
BO(n) := \lim_k G_n(\R^{n+k}), \ \ \ \mbox{and} \ \ \ BO:=\lim_n BO(n)
$$
which are the classifying spaces of the orthogonal group $O(n)$ and
its stable version $\displaystyle O:=\lim_n O(n)$.  The homotopy
groups of $BO$ are well-known:
$$
\pi_q BO= \left\{
\begin{array}{cl}
\Z/2 & \mbox{if} \ q=1,2 \ \mbox{(mod 8)}
\\
\Z  & \mbox{if} \ q=0,4 \ \mbox{(mod 8)}
\\
0  & \mbox{else} 
\end{array}
\right.
$$
Consider the Postnikov  tower of the
space $BO$:
\begin{equation}\label{eq1}
\begin{diagram}
\setlength{\dgARROWLENGTH}{1.2em}
\node{\vdots}
\arrow{s}
\\
\node{BO\!\left< 8\right>}
\arrow{s}
\arrow[2]{e,t}{p_2/6}
\node[2]{K(\Z,8)}
\\
\node{BSpin}
\arrow{s}
\arrow[2]{e,t}{p_1/2}
\node[2]{K(\Z,4)}
\\
\node{BSO}
\arrow{s}
\arrow[2]{e,t}{w_2}
\node[2]{K(\Z/2,2)}
\\
\node{BO}
\arrow[2]{e,t}{w_1}
\node[2]{K(\Z/2,1)}
\end{diagram}
\end{equation}
In each step the lowest homotopy group is killed by the map into the
corresponding Eilenberg-McLane space, and $w_1$, $w_2$ are the
Stiefel-Whitney classes and $p_1$, $p_2$ are the
Pontryagin classes respectively.

Now let $M$ be a manifold, $\dim M =n $. We denote by $h_0$ the
Euclidian metric on $\R^{n+k}$. Then an embedding $j: M \rh \R^{n+k}$
provides $M$ with the Riemannian metric $g=j^*h_0$ induced from the
Euclidian space $\R^{n+k}$. Furthemore, the metric $g$ gives the
tangent and normal bundles $TM$ and $NM$ the Euclidian structure, in
particular, we have the Gauss map
$$
\bar f: M \lra G_k(\R^{n+k}) 
$$
such that $\bar f^*U_{k,n}=NM$ and $\bar f^*U_{k,n}^{\perp}=TM$. A homotopy
class of $\bar f$ depends on the embedding $j: M \rh \R^{n+k}$,
however, it determines uniquely a homotopy class of the composition
$$
f: M \ov{\bar f}{\lra} G_k(\R^{n+k}) \lra BO .
$$
We say that a manifold $M$ has a {\sl string}-structure if the Gauss
map $f: M \lra BO$ lifts to the map $\hat f : M \lra BO\!\left<
  8\right>$, i.e. the following diagram commutes:
\begin{equation}\label{eq2}
\begin{diagram}
\setlength{\dgARROWLENGTH}{1.2em}
\node[4]{BO\!\left< 8\right>}
\arrow{s}
\\
\node[4]{BSpin}
\arrow{s}
\arrow[2]{e,t}{p_1/2}
\node[2]{K(\Z,4)}
\\
\node[4]{BSO}
\arrow{s}
\arrow[2]{e,t}{w_2}
\node[2]{K(\Z/2,2)}
\\
\node{M}
\arrow[3]{ne,t}{\hat f}
\arrow{neee}
\arrow{nneee}
\arrow[3]{e,t}{f}
\node[3]{BO}
\arrow[2]{e,t}{w_1}
\node[2]{K(\Z/2,1)}
\\
\end{diagram} 
\end{equation}
A choice of the lift $\hat f$ is sometimes called a
{\sl string-structure} on $M$. We emphasize that usually we use a
string-structure on the normal bundle $NM$; this
  implies that the tangent bundle $TM$ also has a string structure.
We denote by $\Omega^{\str}_n$ the corresponding cobordism group.

This construction has more general setting.  Let
$BO\!\left<\ell\right>$ be the $(\ell-1)$-connected cover of $BO$. We
say that a manifold $M$ has $BO\!\left<\ell\right>$-structure if it is
given a lift $f_{\left<\ell\right>}$ of the standard Gauss map as
above. Then there is a corresponding cobordism group
$\Omega^{\left<\ell\right>}_n$. Clearly we have that $
\Omega^{\left<4\right>}_n=\Omega^{\spin}_n$, and
$\Omega^{\left<8\right>}_n=\Omega^{\str}_n$.  There is one more
special case when manifolds have $BO\!\left<9\right>$-structure: these
are string manifolds with the vanishing class $\frac{1}{6}p_2$. In
some papers, for instance see \cite{SSS}, manifolds with
$BO\!\left<9\right>$-structure are called as {\sl Fivebrane
  manifolds}, and the cobordism group $\Omega^{\left<9\right>}_n$ is
called {\sl Fivebrane cobordism} and denoted as $\Omega^{\Fbr}_n$.
\section{Surgeries and $BO\!\left<\ell\right>$-manifolds}
 Let
  $M$ be a closed $n$-manifold,  $S^{s}\subset M$ be an embedded
  sphere with trivial normal bundle and let  $t=n-s-1$. This gives an embedding
  $S^s\times D^{t+1}\subset M$ which extends the embedding $S^s\subset
  M$. Then a {\sl surgery along} the embedded $S^s\subset M$ gives the manifold
$$
M'= \left(M\setminus (S^s\times D^{t+1})\right)\cup_{S^s\times S^t}(D^{s+1}\times S^t).
$$
Let $x\in \pi_s(M)$ be an element represented by a map $\xi: S^s \to
M$. If $2s<n$, then according to the Whitney Embedding Theorem,  the map
$\xi$ can be deformed to an embedding $S^s\subset M$. Then we say
that {\sl the element $x\in \pi_s(M)$ can be killed by a surgery} if
such an embedding has trivial normal bundle. 

Let $f: M \to BO$ be the map classifying the stable
  normal bundle of $M$; it gives the induced homomorphism $f_*:
  \pi_s(M)\to \pi_s(BO)$. The following result is well-known (see,
  \cite[Corollary 5.64]{Ranicki}, for example):
\begin{lemma}\label{killing}
  Let $M$ be a smooth manifold and let $f: M\to BO$ be the map classifying
  the stable normal bundle of $M$. Assume $2s<n=\dim M$. Then an
  element $x\in \pi_s(M)$ can be killed by a surgery if and only if
  $f_*(x)=0$ in $\pi_s(BO)$.
\end{lemma}
The previous lemma \ref{killing} implies in particular that for a
$BO\!\left<\ell\right>$-manifold $M$ of dimension $n>2s$, every map
$S^s\to M$ has trivial normal bundle for $s\leq l-1$.  We have
therefore proved the following lemma:
\begin{lemma}\label{lemma1} Assume $n> 2(\ell-k)\geq 2$ and $k\geq 1$.
\begin{enumerate}
\item[{(1)}] Let $x\in \Omega^{\left<\ell\right>}_n$. Then $x$ can
  be represented by some $(\ell-k)$-connected manifold.
\\
\item[{(2)}] Let $M_0$ and $M_1$ be $(\ell-k)$-connected
  $BO\!\left<\ell\right>$-manifolds. Then if
  $M_0$ and $M_1$ represent the same element $x\in
  \Omega^{\left<\ell\right>}_n$, there exists a
  $BO\!\left<\ell\right>$-cobordism $(W,M_0,M_1)$, where the pair
  $(W,M_1)$ is $(\ell-k)$-connected.
\end{enumerate}
\end{lemma}
In particular, if $n> 8$, any element $x\in \Omega^{\str}_n$ can
be represented by a $4$-connected manifold and if $M_0$, $M_1$ are two
$4$-connected manifolds representing $x$, then there exists a
string-cobordism $(W,M_0,M_1)$ where $(W,M_1)$ is $4$-connected.

Next, we need more details on bordisms between 
$BO\!\left<\ell\right>$-manifolds.  We start with the following 
fact which follows from the basic Morse theory:
\begin{lemma}\label{lemma2}
  Let $(W,M_0,M_1)$ be a simply connected bordism,
  $\dim W=n+1$, and let $n \geq p+3$, where $p$ is a
    positive integer.  Assume that $ H_{j} (W,M_1;\Z)=0 $ for all $j\leq
  p$, or that $ H^{j} (W,M_1;\Z)=0 $ for all $j\leq
  p$.  Then $M_1$ can be obtained from $M_0$ by surgeries of
  codimension at least  $p+1$.
\end{lemma}
The following result  is a consequence of Lemma
\ref{lemma2}:
\begin{proposition}\label{theorem1}
  Let $M_1$ be a compact $r$-connected
  $BO\!\left<\ell\right>$-manifold of dimension $n$, where $n\geq
  2r+3$ and $\ell\geq r+2$. Let $M_0$ be a compact
manifold, such that $[M_0]=[M_1]$ in
    $\Omega^{\left<\ell\right>}_n$.  
Then $M_1$ can be
  obtained from $M_0$ by surgeries of codimension at
    least $r+2$.
 \end{proposition}
\begin{proof}
  Let $(W,M_0,M_1)$ be a $BO\!\left<\ell\right>$-cobordism,
and let $M_1$ be $r$-connected.
Using surgeries we can assume that $W$ is $(r+1)$-connected since
$r+1\leq \ell-1$ and the dimension of $W$ is sufficiently
high. Consequently, we have $H_i(W)=0$ for all $i\leq r+1$. On the
other hand $M_1$ is $r$-connected, thus $H_i(M_1)=0$
for all $i\leq r$ and therefore $H_i(W,M_1)=0$ for
all $i\leq r+1$. Finally, Lemma \ref{lemma2} implies that
$M_1$ can be obtained from $M_0$  by
surgeries of codimension $r+2$.
\end{proof}
\begin{remarks}
  {(1)} For $r=1$ and $\ell =4$, Proposition \ref{theorem1} asserts
  the following: if $M_1$ is compact simply connected spin manifold of
  dimension $n\geq 5$ that is spin-cobordant to a manifold $M_0$, then
  $M_1$ can be obtained from $M_0$ by surgeries of codimension at
  least three. This was first noticed by Gromov and Lawson, see
  \cite{GL1}.

    {(2)} For $r=2$ and $\ell =4$, Proposition \ref{theorem1} is the
    same as the surgery Lemma 4.2 in \cite{Labbi1}: if $M_1$ is
    compact $2$-connected manifold of dimension $n\geq 7$ that is
    spin-cobordant to a manifold $M_0$, then $M_1$ can be obtained
    from $M_0$ by surgeries of codimension $\geq 4$.
\end{remarks}
We specify Theorem \ref{theorem1} for string-manifolds.
\begin{corollary}\label{corollary3-co-str}
  Let $M_1$ be a compact $r$-connected string manifold of
  dimension $n\geq 2r+3$, where $r\leq 6$. Then if $M_0$ is a
compact manifold such that $[M_0]=[M_1]$ in $\Omega^{\str}_n$, then
$M_1$ can be obtained from $M_0$ by surgeries of codimensions at
least $ r+2$.

In particular, if $M_1$ is a compact $3$-connected string manifold of
dimension $n\geq 9$ string-cobordant to a manifold $M_0$, then $M_1$
can be obtained from $M_0$ by surgeries of codimensions at least $ 5$.
\end{corollary}
If we continue our climbing of the Postnikov tower we reach the
$BO\!\left<9\right>$-manifolds or {\sl Fivebrane manifolds}.  It is
well known that the corresponding Postnikov invariant is given by
$\frac{1}{6}p_2$, where $p_2$ is the second Pontryagin class.  We
specify Proposition \ref{theorem1} for Fivebrane-manifolds:
\begin{corollary}
  Let $M_1$ be an $r$-connected Fivebrane manifold of dimension $n\geq
  2r+3$, where $r\leq 7$. Then if $M_0$ is a Fivebrane-manifold with
  $[M_0]=[M_1]$ in $\Omega^{\Fbr}_n$, then $M_1$ can be obtained from
  $M_0$ by surgeries of codimensions at least $ r+2$.
 
  In particular, if $M_1$ is a compact $4$-connected Fivebrane
  manifold of dimension $n\geq 11$ that is Fivebrane-cobordant to a
  manifold $M_0$, then $M_1$ can be obtained from $M_0$ by surgeries of
  codimensions at least $ 6$.
\end{corollary}
\subsection{Non-string  3-connected manifolds}
In contrast with the result in Corollary \ref{corollary3-co-str}, we prove the
following result for $3$-connected but non-string manifolds.
\begin{proposition}\label{theorem1bis}
  Let $M_1$ be a $3$-connected and non-string compact manifold of
  dimension $\geq 9$. If $M_1$ is spin cobordant to a manifold $M_0$, then
  $M_1$ can be obtained from $M_0$ by surgeries of codimension $\geq 5$.
\end{proposition}
\begin{proof}  
  Let $(W,M_0,M_1)$ be a spin cobordism, where $M_1$ is $3$-connected,
  non-string with dimension at least 9, and $W$ is
  spin.  Using surgeries we can assume that $W$ is $3$-connected as
  $\dim W=n+1\geq 10$. Consequently by the  Hurewicz theorem we have
  $H_i(W)=0$ for all $i=1,2,3$ and $H_4(W)\cong
  \pi_4(W)$. Similarly, $H_i(M_1)=0$ for $i=1,2,3$ and
  $H_4(M_1)\cong \pi_4(M_1)$ since $M_1$ is $3$-connected.

Since for any $3$-connected space $X$,
  $H^4(X;\Z)\cong \Hom(H_4(X;\Z),\Z)$, the first Pontryagin class
  $p_1(W)$ is given by a homomorphism 
$$
p_1(W): H_4(W;\Z) \lra \Z .
$$
Similarly, the class $p_1(M_1)$ is given by a homomorphism $p_1(M_1):
H_4(M_1;\Z) \lra \Z$.  Then $TW|_N\cong TM_1\oplus \epsilon^1$, where
$\epsilon^1$ is a trivial linear bundle, which implies that
$p_1(W)=i^*(p_1(M_1))$, where $i: M_1\rh W$ is the boundary
inclusion. Also recall that the first Pontryagin class is divisible by
2 for spin manifolds.  Thus we obtain a commutative diagram:
\begin{equation}\label{diag3}
\begin{diagram}
\setlength{\dgARROWLENGTH}{1.5em}
\node{\!\!\!\!\!\!\!\!\!\!\!\!\!\!\!\!\!\!\pi_4(M_1)\cong H_4(M_1)}
\arrow[2]{e,t}{\frac{1}{2}p_1(M_1)}
\arrow{se,t}{i_{*}}
\node[2]{\Z}
\\
\node[2]{\!\!\!\!\!\!\!\!\!\!\!\!\!\!\!\!\!\!\!\!\!\!\!\!\!\!\!\!\!\!\!\!\pi_4(W)\cong H_4(W)}
\arrow{ne,b}{\frac{1}{2}p_1(W)}
\end{diagram}
\end{equation}
{\bf Remark.} Let $H_4(W;\Z)=F_4(W)\oplus T_4(W)$, where $F_4(W)$ and
$T_4(W)$ are free and torsion parts respectively.  Clearly the homomorphism 
$p_1(W): H_4(W;\Z) \lra \Z$, restricted to the torsion part $T_4(W)$, 
is trivial. Thus $p_1(W)$ is not a torsion class, and 
$p_1(W)/2=0$ implies $p_1(W)=0$. 

Both manifolds $W$ and $M_1$ are $3$-connected, spin
  (where, of course, $M_1\subset \p W$), however, they are not
  string-manifolds, i.e. the first Pontryagin class is not zero. We
  would like to show that the kernel of the homomorphism $p_1(W)/2$ can
  be killed by surgeries.
\begin{lemma}\label{4-sphere}
  Let $\eta^k\to S^4$ be a vector bundle of dimension $k\geq
  5$. Then there exists a $4$-dimensional bundle $\xi^4\to S^4$ such
  that $\eta^k\cong \xi^4\oplus \epsilon^{k-4}$, where
  $\epsilon^{k-4}\to S^4$ is a trivial vector bundle.
\end{lemma}
\begin{proof}
  Let $f: S^4 \to BO(k)$ be a map classifying the bundle $\eta$. Then
  we can assume that $S^4$ is mapped to the $4$-th skeleton
  $BO(k)^{(4)}$ of $BO(k)$. It is well-known that $BO(k)^{(4)}\subset
  BO(4)$ if $k\geq 5$. Thus up to homotopy, the map $f$ factors
  through $BO(4)$, i.e. we obtain a commutative (up to homotopy)
  diagram:
$$
\begin{diagram}
\setlength{\dgARROWLENGTH}{1.5em}
\node{S^4}
     \arrow[2]{e,t}{f}
     \arrow{se,b}{f_1}
\node[2]{BO(k)}
\\
\node[2]{BO(4)}
     \arrow{ne,b}{\iota}
\end{diagram}
$$
where $\iota: BO(4)\rh BO(k)$ is the standard embedding. Hence
$\eta^k\cong \xi^4\oplus \epsilon^{k-4}$.
\end{proof}
We continue with the proof of Proposition \ref{theorem1bis}.  Let
$S^4\rh W$ be an embedded sphere representing an element $x\in
\pi_4(W)\cong H_4(W;\Z)$ such that $p_1(x)=0$. We denote by
$\nu_{S^4}$ the normal of the embedding $S^4\rh W$. By
assumption,$TW|_{S^4}$ is stably trivial, i.e.
$TW|_{S^4}\oplus\epsilon^{k-n}\cong \epsilon^k$ for some
$k>n$. However, $TW|_{S^4}\cong TS^4\oplus \nu_{S^4}$, and we have
that
$$
 TS^4\oplus \nu_{S^4} \oplus \epsilon^{k-n}\cong \epsilon^k
$$
Since $TS^4\oplus \epsilon^1$ is a trivial bundle, thus 
$$
\epsilon^5\oplus\nu_{S^4}\oplus \epsilon^{k-n-1}\cong \epsilon^k.
$$
In particular, we obtain that $p_1\nu_{S^4}=0$, and Lemma
\ref{4-sphere} gives that $\nu_{S^4}=\xi^4\oplus\epsilon^{n-4}$ where
$\xi^4$ is a $4$-dimensional bundle with $p_1\xi^4=0$. Thus $\xi^4$ is
trivial bundle, i.e. the normal bundle $\nu_{S^4}$ is trivial.

Thus we can use surgeries on $W$ to kill the kernel
of $\frac{1}{2}p_1(W)$ and therefore one can uses 
the long exact sequence,
\begin{equation*}
\begin{diagram}
\setlength{\dgARROWLENGTH}{1.5em}
\node{H_4(M_1)}
\arrow{e}
\node{H_4(W)}
\arrow{e}
\node{H_4(W,M_1)}
\arrow{e}
\node{H_3(M_1)}
\end{diagram}
\end{equation*}
to show that $H_4(W,M_1)=0$ up to torsion elements.  In particular
$H^4(W,M_1)=0$. Therefore
Lemma \ref{lemma2} implies that $M_1$ can be obtained
  from $M_0$ by surgeries of codimension at least $5$.
\end{proof}
\subsection{Non-Fivebrane $7$-connected manifolds}
It is well-known that the second
Pontryagin class of a string manifold is divisible by
6 and $\frac{1}{6}p_2$ serves as the obstruction to
lifting a string structure to a Fivebrane structure.  One can without
difficulties adapt the proof of Proposition
  \ref{theorem1bis} to show the following:
\begin{proposition}\label{theorem1bisbis}
  Let $N$ be a $7$-connected and non-Fivebrane compact manifold of
  dimension $\geq 15$. If $N$ is string cobordant to a manifold $M$,
  then $N$ can be obtained from $M$ by surgeries of codimension $\geq
  9$.
\end{proposition}
\section{Positive curvature and Surgeries}\label{Newsection3} 
The results of the previous section suggest that geometrical
properties that are stable under surgeries should have a
topological interpretation. This is the case for the positivity of the
$p$-curvatures and the second Gauss-Bonnet curvature as we will see in
the rest of this paper.
\subsection{Positive $p$-curvature}
We denote by $G_p(\R^{n})$ the
Grassmanian manifold of $p$-dimensional subspaces in $\R^{n}$, as
above.  Let $(M,g)$ be a Riemannian manifold. Then the metric $g$
provides the tangent bundle $TM$ the structure group $O(n)$.
This gives an associated smooth bundle
\begin{equation}\label{Grassm}
G_p(TM):=TM\times_{O(n)}G_p(\R^{n}) \lra M,
\end{equation}
with the fiber $G_p(TM_x)\cong G_p(\R^{n})$ over $x\in M$ and
the structure group $O(n)$. 

Then  the
$p$-curvature $s_p$, for $0\leq p\leq n-2$,  is a function
$$
s_p : G_p(TM) \lra \R 
$$
defined as follows.  Let $V$ be a tangent $p$-plane
at $x\in M$. We choose an orthonormal basis $\{e_i\}$ of the
orthogonal complement $V^{\perp}$ of
$V$ in $TM_x$, and define
\begin{equation}\label{p-curv}
  s_p(V)= \sum_{i,j=p+1}^{n}K_{i,j},
\end{equation}
where $K_{i,j}=K(e_i,e_j)$ is the usual sectional curvature. The
$0$-curvature $s_0$ coincides with the usual scalar
curvature $\Scal$, the $1$-curvature is the Einstein curvature and the
$(n-2)$-curvature is the usual sectional curvature.  

We are interested in understanding the conditions under which a manifold 
admits a Riemannian metric $g$ with positive $p$-curvature. 

We emphasize that if $s_{p}>$ then $s_j>0$ for all $j<p$. 
It turns out that the positivity of the $p$-curvature
can be preserved under surgeries:
\begin{theorem}[{\cite{Labbi1}}]\label{surgery1}
  Let $g_0$ be a Riemannian metric on a compact manifold $M_0$ with $s_p(g_0)>0$, and
  $M_1$ be a manifold constructed out of $M_0$ by a surgery of
  codimension $\geq p+3$. Then there exists a Riemannian metric $g_1$
  on $M_1$ with $s_p(g_1)>0$.
\end{theorem}
There is a natural generalization of Theorem \ref{surgery1} for
elementary cobordism:
\begin{theorem}\label{surgery2}
  Let $g_0$ be a Riemannian metric on a compact manifold $M_0$ with $s_p>0$, and $M_1$ be
  a manifold constructed out of $M_0$ by a surgery of codimension
  $\ell+1\geq p+3$. Let
$$
W= M_0\times I\cup (D^{k+1}\times
  D^{\ell+1}), \ \ \p W = M_0\sqcup M_1,
$$ 
be the corresponding elementary cobordism.  Then there exists a
Riemannian metric $\bar g $ on $W$ with positive $p$-curvature and
such that
$$
\left\{
\begin{array}{ll}
\bar g = g_0 + dt^2 & \mbox{near $M_0$} ,
\\
\bar g = g_1 + dt^2 & \mbox{near $M_1$} .
\end{array}\right.
$$
In particular, the $p$-curvature of the metric $g_1$ is positive.
\end{theorem}
\begin{remark}
The methods of M. Walsh's papers \cite{Walsh1,Walsh2} 
can be  adapted to provide a proof for the above  Theorem \ref{surgery2} and also  for Theorem \ref{surgery3} below.
\end{remark}
\subsection{Positive second Gauss-Bonnet curvature}
For a given Riemannian manifold $(M,g)$, we denote by
$R$, $\Ric$ and $\Scal$ respectively the Riemann curvature tensor, the
Ricci curvature tensor and the scalar curvature.  The {\sl second
  Gauss-Bonnet curvature}, denoted by $h_4$, is a quadratic scalar
curvature and it is defined by 
\begin{equation*}
\begin{array}{c}
h_4=||R||^2-||\Ric||^2+\frac{1}{4}\Scal^2,
\end{array}
\end{equation*}
see \cite{Labbi2}. Note that in four dimensions, the curvature $h_4$ coincides with the Gauss-Bonnet integrand.  This curvature is considered by
physicists as a possible substitute to the usual scalar curvature to
describe gravity in higher general theories of relativity, for
instance in string theories. Here we are interested in the positivity
properties of this invariant. First, let us recall the following
stability under surgeries result:
\begin{theorem}[\cite{Labbi2}]\label{surgery2a}
  Let $g_0$ be a Riemannian metric on a compact manifold $M_0$ with
    $h_4(g_0)>0$, and $M_1$ be a manifold constructed out of $M_0$ by
    a surgery of codimension at least $ 5$. Then there exists a
    Riemannian metric $g_1$ on $M_1$ with $h_4(g_1)>0$.
\end{theorem}

There is a natural generalization of Theorem \ref{surgery2a}:

\begin{theorem}\label{surgery3}
  Let $g_0$ be a Riemannian metric on a compact manifold $M_0$ with $h_4(g_0)>0$, and
  $M_1$ be a manifold constructed out of $M_0$ by a surgery of
  codimension $\ell+1\geq 5$. Let
$$
W= M_0\times I\cup (D^{k+1}\times
  D^{\ell+1}), \ \ \p W = M_0\sqcup M_1,
$$ 
be the corresponding elementary cobordism.  Then there exists a
Riemannian metric $\bar g $ on $W$ with $h_4(\bar g)>0$  and
such that
$$
\left\{
\begin{array}{ll}
\bar g = g_0 + dt^2 & \mbox{near $M_0$} ,
\\
\bar g = g_1 + dt^2 & \mbox{near $M_1$} .
\end{array}\right.
$$
In particular, $h_4(g_1)>0$.
\end{theorem}
\begin{remarks}
\begin{enumerate}
\item Theorems \ref{surgery2a} and \ref{surgery3} are still valid if
  we require the metric $g_0$ to have positive $h_4$ and positive
  $2$-curvature at the same time.
\item Theorems \ref{surgery2a} and \ref{surgery3} can be also be
  generalized to all higher Gauss-Bonnet curvatures $h_{2k}$. It is
  plausible that the condition $h_{2k}>0$ could be preserved under
  surgeries of codimension at least $2k+1$.
\item
Because $h_4$ is quadratic in the Riemann curvature tensor, one can
expect that the condition $h_4>0$ has two components, where each
component is an inequality that is linear in the Riemann curvature
tensor, it would be interesting to determine these components. 
In this direction, Theorem \ref{positivity} below 
relates the positivity of $h_4$ to the positivity and negativity of
the $p$-curvatures.
\end{enumerate}
\end{remarks}
\begin{theorem}[\cite{Labbi2}]\label{positivity}
  Let \ $(M,g)$ \ be a Riemannian manifold of
    dimension $n\geq 4$.  Assume that $s_p(g)\geq 0$ or $s_p(g)\leq 0$
    (respectively, $s_p(g)> 0$ or $s_p(g)< 0$), where $p\geq
    \frac{n}{2}$. Then $h_4(g)\geq  0$ (respectively, $h_4(g)> 0$).
  Furthermore, $h_4(g)\equiv 0$ if and
    only if the manifold $(M,g)$ is flat.
\end{theorem}
\section{First Applications: Proof of Theorems A, B and
  B$'$}\label{section3}
\subsection{Proof of theorem A}
Theorem A is a consequence of the following theorem
\begin{theorem}\label{thm-Av}
  Let $M_1$ be a compact $3$-connected manifold of dimension $\geq 9$
  which is not string. If $M_1$ is spin cobordant to a manifold $M_0$
  which carries a metric $g_0$ with $s_2(g_0)>0$ (respectively, with
  $h_4(g_0)>0$), then $M_1$ also carries a metric $g_1$ with
  $s_2(g_1)>0$ (respectively, with $h_4(g_1)>0$).
 \end{theorem}
\begin{proof}
  On one hand, Proposition \ref{theorem1bis} shows that the manifold
  $M_1$ can be obtained from $M_0$ by surgeries of codimension at
  least $5$. On the other hand, since $M_0$ has positive $2$-curvature
  (resp. positive $h_4$ curvature), Theorems \ref{surgery1} and
  \ref{surgery2a} show therefore that $M_1$ also carries a metric with
  $s_2>0$ (respectively, with $h_4 >0$).
\end{proof}
Now we prove Theorem A as follows. A compact $3$-connected manifold
$M$ of positive $2$-curvature and with dimension $\geq 9$  is in
particular a simply connected manifold of positive scalar curvature
and therefore its $\alpha$-genus vanishes by Theorem
\ref{GL-thm2}. Conversely, compact non-string $3$-connected manifold
$M$ of dimension $\geq 9$  and with vanishing $\alpha$-genus is spin
cobordant to the total space $E$ of an $\H\P^2$-bundle by Theorem
\ref{stolz}. The total space $E$ has positive $2$-curvature and
positive $h_4$ curvature. Thus the above Theorem \ref{thm-Av} shows
that $M$ carries a metric with positive $2$-curvature and a metric
with $h_4>0$.
\subsection{Proof of Theorems B and B$'$}
Now we return to $BO\!\left<\ell\right>$-manifolds. The following
theorem unifies and generalizes at the same time Theorem B of
\cite{GL1}, Lemma 4.2 of \cite{Labbi1} and Theorems B and B' of this
paper that were stated in the introduction.
\begin{theorem}\label{application} 
  Let $n,r, \ell$ be positive integers such that $n\geq 2r+3$ and
  $\ell\geq r+2$. Let $M_1$ be a compact $r$-connected
  $BO\!\left<\ell\right>$-manifold of dimension $n$. Assume
  $[M_1]=[M_0]$ in the cobordism group $\Omega_n^{\left<\ell\right>}$,
  where $M_0$ admits a Riemannian metric $g_0$ with $s_p(g_0)>0$ for
  some $p$ such that $0\leq p\leq r-1$.  Then $M_1$ also admits a
  Riemannian metric $g_1$ with $s_p(g_1)>0$.
\end{theorem}
\begin{proof}
  Let $0\leq p\leq r-1$ be as in the theorem.  Proposition
  \ref{theorem1} asserts that the manifold $M_1$ can be obtained from
  $M_0$ using surgeries of codimensions $\geq r+2\geq p+3$.  Since the
  manifold $M_0$ is supposed to have positive $p$-curvature then
  Theorem \ref{surgery1} shows that $M_1$ carries as well a metric
  with $s_p>0$.
\end{proof}
Note that we recover Theorem B of \cite{GL1} about the scalar
curvature (that is the $0$-curvature) when $p=0, r=1$ and
$\ell=4$. Lemma 4.2 of \cite{Labbi1} about the $1$-curvature is
obtained for for $p=1, r=2$ and $\ell=4$.Theorems B and B' of this
paper are respectively obtained for $\ell =8$ and $\ell=9$.
\begin{remarks}
\begin{itemize}
\item[(1)] Similar  results hold if we replace
    positive $2$-curvature by positive $h_4$-scalar curvature,
    for instance:
{\it A compact $3$-connected string manifold of dimension $n\geq 9$
    that is string cobordant to a manifold of positive $h_4$ has a
    metric with positive $h_4$ curvature.}
\item[(2)]  Let $n=11$ or $n=13$. Since in these particular dimensions string
    $n$-manifolds are known to be cobordant to zero we conclude that a
    compact $3$-connected string $n$-manifold always has a metric with
    positive $2$-curvature and a metric with $h_4$ positive. Similar
    results hold for the $p$-curvatures for $p\leq 5$ as above.
\end{itemize}
\end{remarks}
\subsection{Genera for string manifolds and positive curvature}
Recall that the string cobordism ring
$\Omega^{\str}_*=\bigoplus_{n\geq 0} \Omega_n^{\str}$ is is the ring
whose elements of order $n$ are string-cobordism classes of
$n$-dimensional string manifolds, the addition operation is given by
the disjoint union of manifolds and product operation is given by the
Cartesian product of manifolds.

Let $I_1$ (resp. $I_2$, $I_3$) denote the subset of $\Omega^{\str}_*$
which consists of bordism classes containing representatives with
positive $2$-curvature (resp. positive $h_4$, positive $h_4$ and
positive $s_2$). Since the cartesian product of a manifold of positive
$2$-curvature (resp. positive $h_4$, positive $h_4$ and positive
$s_2$) with an arbitrary manifold has positive $2$-curvature
(resp. positive $h_4$, positive $h_4$ and positive $s_2$), we conclude
that $I_1$ (resp. $I_2$, $I_3$) is an ideal of $\Omega^{\str}_*$. We
therefore get the following three genera (ring homomorphisms):
\begin{equation}
\Pi_i: \Omega^{\str}_*\rightarrow \Omega^{\str}_*/I_i,
\end{equation}
for $i=1,2,3$. A natural question arises at this level: Are the
previous three (geometrical) genera topological genera? Are they for
instance related to Witten genus?

Recall that for a string manifold $N$, the Witten genus, denoted
$\phi_W(N)$, is a modular form for $SL_2(\Z)$ with integer
coefficients.  In particular, the Witten genus $\phi_W$ defines a ring
homomorphism from the bordism ring $\Omega^{\str}_*$ to the ring of
integral modular forms for $SL_2(\Z)$, see the next section for more
details.

We shall prove in section \ref{sec-app} that $\Ker\ \phi_W \otimes
\Q\subset I_i$ for $i=1,2,3$. It remains an open question to decide
whether the previous inclusion is in fact an equality, that is a
vanishing theorem of Lichn\'erowicz type: If $N$ is a string manifold
of positive $2$-curvature (resp. positive $h_4$, positive $h_4$ and
positive $s_2$) then $\phi_W(N) = 0.$

An important question in the same direction is Stolz's conjecture
\cite{SSS}:
\begin{conj-s}
 If $N$ is a string manifold and admits a positive Ricci
curvature metric, then $\phi_W(N) = 0.$
\end{conj-s}
At this time no counterexamples are known to this conjecture and the
conjecture is proven to be true for some classes. However, these
classes admit also metrics with different positivity properties, for
instance metrics with positive $p$-curvature, and so it may be
possible that the Stolz conjecture is true for positive $p$-curvature.

Let us note here that all the known constructions of string manifolds
with positive $p$-curvature through group actions and Riemannian
submersions \cite{Labbigroups}, or through surgeries \cite{Labbi1}
have vanishing Witten genus. This is a consequence of a result due to
Dessai, H\"ohn and Liu \cite{ Dessai,Liu} where they prove the
vanishing of the Witten genus on connected string manifolds with
non-trivial smooth $S^3$-actions and of another related result of
Dessai \cite{Dessai}. The later asserts the vanishing of the Witten
genus on any smooth fibre bundle of closed oriented manifolds provided
the fibre is a string manifold and the structure group is a compact
connected semi-simple Lie group which acts smoothly and non-trivially
on the fibre.
\section{Witten genus and its kernel in $\Omega^{\str}_*\otimes
  \Q$}\label{witten-genus}
\subsection{Cayley projective plane} Here we recall necessary facts
concerning the Cayley projective plane $\Ca\P^2$. We follow the
constructions due to A. Dessai \cite{Dessai}.  Let $F_4$ denote the
$52$-dimensional compact simple sporadic Lie group. It is well-known
that $F_4$ contains a group isomorphic to $Spin(9)$ which is unique up
to inner automorphism of the ambient group $F_4$. We choose such a
subgroup and identify it with $Spin(9)$. Then we can identify the
Cayley projective plane $\Ca\P^2$ with the homogeneous space
$F_4/Spin(9)$. This is $7$-connected smooth manifold with the
cohomology ring $H^*(\Ca\P^2;\Z)\cong \Z[z]/z^3$, where $z\in
H^8(\Ca\P^2;\Z)$ is a generator.  In particular, $\Ca\P^2$ is a fiber
of the fiber bundle
$$
BSpin(9)\to BF_4
$$
induced by the embedding $Spin(9)\subset F_4$. The bundle $BSpin(9)\to
BF_4$ is a universal {\sl geometric $\Ca\P^2$-bundle}. 

Let $L$ be a smooth manifold, $\dim L=\ell$, and $f : L \to BF_4$ be a
map. Then one obtains the induced bundle with the fiber $\Ca\P^2$ and
structure group $F_4$:
\begin{equation*}
\begin{diagram}
\setlength{\dgARROWLENGTH}{1.5em}
\node{W}
     \arrow[2]{e,t}{f^*}
     \arrow{s,r}{\pi}
\node[2]{BSpin(9)}
     \arrow{s}
\\
\node{L}
     \arrow[2]{e,t}{f}
\node[2]{BF_4}
\end{diagram}
\end{equation*}
Let $T_{\spin}\subset Spin(9)$ be the maximal torus covering the
maximal torus $T_{\so}\subset SO(9)$. It is convenient to choose a
basis $\hat\xi_1,\hat\xi_2,\hat\xi_3,\hat\xi_4$ of the Lie algebra of
$T_{\so}$ which is also a basis for the Lie algebra of
$T_{\spin}$. Then the integral lattice in $\R^4$ which provides the
universal cover of the torus $T_{\spin}$ is given by
$a_1\hat\xi_1+\cdots+a_4\hat\xi_4$, where the sum $a_1+\cdots+a_4$ of
integers is even. Let $\hat\xi$ be a generator of the Lie algebra of
$S^1$, and $v: S^1\to T_{\spin}$ be such a map for which the
differential $dv$ takes $\hat\xi$ to $2\hat\xi_1$. Then the
composition
$$
\hat v : S^1\st{v}{\lra}T_{\spin}\lra Spin(9)\lra F_4
$$
induces a map $B\hat v : BS^1 \to BF_4$. We obtain
the following diagram of fiber bundles:
\begin{equation*}
\begin{diagram}
\setlength{\dgARROWLENGTH}{1.5em}
\node{E}
     \arrow[2]{e,t}{(B\hat v)^*}
     \arrow{s,r}{\pi}
\node[2]{BSpin(9)}
     \arrow{s}
\\
\node{BS^1}
     \arrow[2]{e,t}{B\hat v}
\node[2]{BF_4}
\end{diagram}
\end{equation*}
where the bundle $\pi: E\to BS^1$ has the fiber $\Ca\P^2$ and the structure
group is reduced from the group $F_4$ to its subgroup $v(S^1)\subset F_4$. 

Then one can choose a subgroup $S^3$ of the centralizer of the group
$v(S^1)$ in $F_4$ so that $S^3$ acts nontrivially on the orbit space
$\Ca\P^2=F_4/Spin(9)$. A particular choice is given by the subgroup
$S^3\cong Spin(3)\subset Spin(9)$ which covers the subgroup 
$$
\left(
\begin{array}{c|c}
1 & 0
\\  \hline
0 & SO(3)
\end{array}
\right) \subset SO(9)
$$
under the canonical covering map $Spin(9)\to SO(9)$. Then the subgroup
$S^3$ commutes with the structure group $v(S^1)$ of the fiber bundle
$\pi: E\to BS^1$.  Thus we obtain a nontrivial action of $S^3$ along
the fibers $\Ca\P^2=F_4/Spin(9)$ of the total space $E$. 

Assume that a map $f:L \to BF_4$ is given by a composition $L\st{h}{\lra} BS^1
\st{Bv}{\lra} BF_4$. Then the geometric $\Ca\P^2$-bundle $W\to L$ is
given by the diagram of fiber bundles
\begin{equation}\label{induced}
\begin{diagram}
\setlength{\dgARROWLENGTH}{1.5em}
\node{W}
     \arrow[2]{e,t}{h^*}
     \arrow{s}
\node[2]{E}
     \arrow[2]{e,t}{(B\hat v)^*}
     \arrow{s,r}{\pi}
\node[2]{BSpin(9)}
     \arrow{s}
\\
\node{L}
     \arrow[2]{e,t}{h}
\node[2]{BS^1}
     \arrow[2]{e,t}{B\hat v}
\node[2]{BF_4}
\end{diagram}
\end{equation}
In particular, the structure group of the bundle $W\to L$ is reduced
to $v(S^1)\subset F_4$, and there is a non-trivial fiber-wise action
of $S^3$ on $W$.  This construction leads to the following result:
\begin{proposition}\label{Dessai-1}{\rm (A. Dessai, \cite[Proposition
    5.2]{Dessai})} There exist oriented manifolds $M^{4k}$ such that 
$$
\Omega_*^{\so}\otimes \Q \cong \Q[x_4,\ldots,x_{4k},\ldots],
$$
where $x_{4k}=[M^{4k}]_{\so}$, $k=1,2,\ldots$, and 
$M^{4k}$ satisfy the following conditions:
\begin{enumerate}
\item[(a)] $M^{4k}$ is a simply connected spin manifold for all $k\geq 1$;
\item[(b)] $M^{4k}$ is a string manifold for all $k\geq 2$;
\item[(c)] $M^{4k}$ is the total space of a geometric $\Ca\P^2$-bundle
  with structure group $S^1$ and non-trivial $S^3$-action along the
  fibers for $k\geq 4$.
\end{enumerate}
\end{proposition}
We recall key points on the construction of manifolds $M^{4k}$ given
by A. Dessai \cite{Dessai}. The manifold $M^4$ can be chosen as the
$K_3$-surface given by the quartic $x_0^4+\cdots+x_3^4=0$ in $\C\P^3$,
and the manifolds $M^8$ and $M^{12}$ as almost parallelizable
manifolds with non-vanishing top Pontyagin class. These manifolds can
be constructed by means of plumbing, see \cite{KM}. The manifold
$M^{16}=\Ca\P^2$ is the Cayley projective plane. For $k\geq 5$, the
manifold $M^{4k}$ can be chosen as the total space of geometric
$\Ca\P^2$-bundle over a complete intersection $L_{k-4}\subset
\C\P^{2k-4}$ of complex dimension $2k-8$. The manifold $L_{k-4}$ comes
together with a nontrivial class $c\in H^2(L_{k-4};\Z)$ which is the
first Chern class of the restriction of the dual Hopf bundle over
$\C\P^{2k-4}$. Then for carefully chosen integer $a_{k-4}$, the class
$a_{k-4}c\in H^2(L_{k-4};\Z)$ gives a map $h_k: L_{k-4}\to BS^1$ such
that the induced bundle $W\to L_{k-4}$ given by (\ref{induced}), where
we let $L=L_{k-4}$ and $h=h_k$, is a geometric $\Ca\P^2$-bundle with
structure group $S^1$ and non-trivial $S^3$-action along the fibers.

Let $M$ be a $\spin$ manifold, $\dim M=4k$. Then the $\hat{A}(M)$ is
well-defined and coincides with the index of the standard Dirac
operator on $M$. For any real vector bundle $V$ over $M$, we denote by
$\hat{A}(M;V)$ the index of the Dirac operator on $M$ twisted by the
complexified vector bundle $V\otimes \C$. A {\sl total symmetric
  power} $S_t(V)$ of a vector bundle $V$ is given as a series
$$
S_t(V) := 1+ S^1(V)t+S^2(V)t^2+\cdots \ ,
$$
where $S^j(V)$ is the $j$-th symmetric power of $V$ and $t$ is an
indeterminate varaible. Consider the tensor product
$$
{\mathbb S}(V):=\bigotimes_{m=1}^{\infty} S_{q^m}(V) = 1 + V
q+\left(S^2(V)\oplus V\right)q^2+\left(S^3(V)\oplus (V\otimes
  V)\oplus V\right) q^3+\cdots,
$$
see \cite[Section 2]{Stolz}. Then the {\sl Witten genus} $\phi_W(M)$
(where $\dim M=4k$) is defined as the series
$$
\begin{array}{lcl}
  \phi_W(M)\!\!\!  & = & \!\!\! \displaystyle
  \hat{A}(M;{\mathbb S}(V)) \cdot
  \prod_{n=1}^{\infty}(1-q^n)^{4k} 
  \\
  \\
  & = & \!\!\! \displaystyle
  \left(1\! +\! V
    q\!+\!\left(S^2(V)\oplus V\right) q^2\!+\!
    \cdots\right)\!\cdot\!
  \prod_{n=1}^{\infty}(1-q^n)^{4k} \ , 
\end{array}
$$
see \cite[Section 2]{Stolz} or \cite[Section 2]{Dessai}. It is easy to
see that $\phi_W(M) \in \Z[[q]]$, and $\hat{A}(M)$ is the constant
term of the series $\phi_W(M)$. In particular, the Witten genus
defines the homomorphism
$$
\phi_W : \Omega_*^{\str}\lra \Z[[q]].
$$
Proposition \ref{Dessai-1} implies the following result
\begin{corollary}\label{Dessai-2} {\rm (A. Dessai, \cite{Dessai})}
\begin{enumerate}
\item[{\bf (1)}] There is an isomorphism $
  \Omega^{\str}_*\otimes\Q\cong
  \Q[x_8,x_{12},x_{16},\ldots,x_{4k},\ldots], $ where
  $x_{4k}=[M^{4k}]_{\str}$, and the string manifolds $M^{4k}$ are as
  in Proposition {\rm \ref{Dessai-1}.}
\item[{\bf (2)}] The kernal $(\Ker\ \phi_W)\otimes \Q\subset
  \Omega^{\str}_*\otimes\Q$ coincides with the ideal generated by the
  elements $x_{4k}$, $k\geq 4$.
\item[{\bf (3)}] If $x\in \Ker\ \phi_W\subset\Omega^{\str}_*$, then
  some multiple of $x$ can be realized as the total space of a geometric
  $\Ca\P^2$-bundle.
\end{enumerate}
\end{corollary}
We emphasize the the string manifolds $M^8$ and $M^{12}$ have
non-trivial Witten genus just because $\hat A(M^8)$ and $\hat
A(M^{12})$ are non-zero by construction. 
%
\section{Further Applications: Proof of Theorem A$'$ and  Corollary B}\label{sec-app}

The previous corollary \ref{Dessai-2} asserts in particular that if
$N$ is a string manifold with vanishing Witten genus then a non-zero
multiple of $N$ is string cobordant to a string manifold which is the
total space of a $\Ca\P^2$ bundle with structure group $S^1$ and
non-trivial $S^3$-action along the fibres.

On the other hand the Cayley projective plane $\Ca\P^2$ has dimension
$16$ and positive sectional curvature. In particular, using a result
of \cite{Labbigroups}, the total spaces of $\Ca\P^2$ bundles have
positive $p$-curvature for $0\leq p\leq 14$ (and as well positive
$h_4$-curvature). Corollary B results therefore immediately from
Theorem B.

Next, we prove Theorem A$'$.  Let $N$ be a $7$-connected and
non-Fivebrane compact manifold of dimension  $\geq 15$. Assume that $N$
is string-cobordant to a manifold $M$ which carries a metric with
positive $6$-curvature. Proposition \ref{theorem1bisbis} shows that
the manifold $N$ can then be obtained by performing surgeries on $M$
of codimension $\geq 9\geq 6+3$. Theorem \ref{surgery1} implies then
that $N$ carries a metric of positive $6$-curvature.

Finally, the manifold $N$ is $7$-connected so it is a string manifold.
Since the Witten genus of $N$ is zero then by corollary
\ref{Dessai-2},  a non-zero multiple of $N$ is string cobordant to
a string manifold which is the total space of a $\Ca\P^2$ bundle. As
above the total spaces of $\Ca\P^2$ bundles have positive
$6$-curvature, we deduce therefore from the first part of this theorem
that some multiple $N\sharp ...\sharp N$ carries a metric of positive
$6$-curvature.


\bibliographystyle{amsplain}

\address{Boris Botvinnik\\
305 Fenton Hall, Department of Mathematics,\\
University of Oregon,\\ 
Eugene OR 97403-1222, U.S.A.}\\
\email{botvinn@math.uoregon.edu}
\urladdr{http://pages.uoregon.edu/botvinn/}

\address{Mohammed Larbi Labbi\\
Mathematics Department, College of Science\\
University of Bahrain\\
32038 Bahrain.}\\
\email{labbi@sci.uob.bh}
\urladdr{http://sites.google.com/site/mllabbi/}

\end{document}